\documentclass[12pt,oneside,english]{amsart}

\usepackage[latin9]{inputenc}
\usepackage{geometry}
\usepackage{color}
\usepackage{babel}
\usepackage{amstext}
\usepackage{amsthm}
\usepackage{amssymb}
\usepackage[unicode=true,
 bookmarks=true,bookmarksnumbered=false,bookmarksopen=false,
 breaklinks=false,pdfborder={0 0 0},pdfborderstyle={},backref=false,colorlinks=true,linkcolor=blue,citecolor=blue]
 {hyperref}

\makeatletter
\numberwithin{equation}{section}
\numberwithin{figure}{section}
\theoremstyle{plain}
\newtheorem{thm}{\protect\theoremname}[section]
\theoremstyle{plain}
\newtheorem{cor}[thm]{\protect\corollaryname}
\theoremstyle{plain}
\newtheorem{prop}[thm]{\protect\propositionname}
\theoremstyle{plain}
\newtheorem{lem}[thm]{\protect\lemmaname}
\theoremstyle{remark}
\newtheorem{rem}[thm]{\protect\remarkname}
\theoremstyle{remark}
\newtheorem*{rem*}{\protect\remarkname}
\theoremstyle{remark}
\newtheorem*{acknowledgement*}{\protect\acknowledgementname}

\usepackage{bbm}

\newcommand{\N}{\mathbb{N}}

\newcounter{mythmcounter}
\newtheorem{mythm}{\bf Theorem}[mythmcounter]

\newcommand{\1}{{\large\mathbb{1}}}
\usepackage{mathbbol}

\makeatother

\providecommand{\acknowledgementname}{Acknowledgement}
\providecommand{\corollaryname}{Corollary}
\providecommand{\lemmaname}{Lemma}
\providecommand{\propositionname}{Proposition}
\providecommand{\remarkname}{Remark}
\providecommand{\theoremname}{Theorem}

\usepackage{color}

\begin{document}
\title[Fefferman-Stein inequalities for the maximal function on the $k$-ary
tree]{Fefferman-Stein inequalities for the Hardy{-}Littlewood maximal function
on the infinite {rooted} $k$-ary tree}
\author{Sheldy Ombrosi}
\author{Israel P. Rivera-Ríos}
\author{Martín D. Safe}
\address{Departamento de Matemática, Universidad Nacional del Sur (UNS), Bahía
Blanca, Argentina and INMABB, Universidad Nacional del Sur (UNS)-CONICET,
Bahía Blanca, Argentina.}
\email{sombrosi@uns.edu.ar (Sheldy Ombrosi) }
\email{israel.rivera@uns.edu.ar (Israel P. Rivera-Ríos)}
\email{msafe@uns.edu.ar (Martín D. Safe)}
\thanks{The first and the second author were partially supported by grant
PIP (CONICET) 11220130100329CO. The third author was partially supported
by ANPCyT PICT 2017-13152 and UNS Grant PGI 24/L115}
\begin{abstract}
In this paper weighted endpoint estimates for the Hardy-Littlewood
maximal function on {the infinite rooted} $k$-ary tree are provided. Motivated
by Naor and Tao \cite{NT} the following Fefferman-Stein estimate
\[
w\left(\left\{ x\in T\,:\,Mf(x)>\lambda\right\} \right)\leq c_{s}\frac{1}{\lambda}\int_{T}|f(x)|M(w^{s})(x)^{\frac{1}{s}}dx\qquad s>1
\]
is settled and moreover it {is shown it} is sharp, in the sense that it does not
hold in general if $s=1$. Some examples of non trivial weights such
that the weighted weak type $(1,1)$ estimate holds are provided.
A {strong} Fefferman-Stein type estimate and as a consequence some
vector valued extensions are obtained. In the Appendix a weighted
counterpart of the abstract {theorem} of Soria and Tradacete on infinite
trees \cite{STInf} is established.
\end{abstract}

\maketitle

\section{Intoduction and main results}

The centered Hardy{-}Littlewood maximal function on $\mathbb{R}^{d}$
is defined as 
\[
Mf(x)=\sup_{r>0}\frac{1}{|B(x,r)|}\int_{B(x,r)}|f(y)|dy.
\]
Due to fact that the Lebesgue measure is doubling, namely, {since} 
\[
|B(x,2r)|\leq2^{d}|B(x,r)|
\]
it is not hard to check that $Mf(x)\simeq M^{u}f(x)$ where 
\[
M^{u}f(x)=\sup_{x\in B}\frac{1}{|B|}\int_{B}|f(y)|dy
\]
and $B$ is any ball. Furthermore we may replace balls by cubes with
their sides parallel to the {axes}. 

This operator was shown to be bounded on $L^{p}$ and of weak type
$(1,1)$ by Hardy and Littlewood \cite{HL1930} in the case $d=1$
and by Wiener \cite{W1939} for the case $d\geq1$. In a pioneering
work by Fefferman and Stein \cite{FS1971} the following two weights
inequality was provided
\begin{equation}
w\left(\left\{ x\in\mathbb{R}^{d}\,:\,Mf(x)>t\right\} \right)\lesssim_{d}\frac{1}{t}\int_{\mathbb{R}^{d}}|f(x)|Mw(x)dx.\label{eq:FS}
\end{equation}
Inequality (\ref{eq:FS}) is important for several reasons. {The first}
of them is that it was a cornerstone to provide vector valued extensions.
Another fundamental reason is that it was a precursor of the theory
of weights that was continued later {by} the seminal work by Muckenhoupt~\cite{M1972}. We recall that in the classical setting $w\in A_{1}$
if $\left\Vert \frac{Mw}{w}\right\Vert _{L^{\infty}}<\infty$. Since,
in general, $w\leq Mw$, this condition implies that actually $w\simeq Mw$.
Note that if $w\in A_{1}$ from (\ref{eq:FS}) it readily follows
that 
\[
w\left(\left\{ x\in\mathbb{R}^{d}\,:\,Mf(x)>t\right\} \right)\lesssim_{d}\frac{1}{t}\int_{\mathbb{R}^{d}}|f(x)|w(x)dx.
\]
At this point we would also like to note that Fefferman and Stein
noted in \cite{FS1971} that $w\in A_{1}$ is a necessary condition
for this inequality to hold. Since those works, the theory of weights
and more in particular Fefferman-Stein inequalities and related variants
have been studied in a variety of contexts \cite{OP,To,ABN,R19} and
for singular integrals \cite{CF1976,P1994,MR2923171,DSLR2016,CLO2017,LiPRR}
and their commutators \cite{P1997,PP2001,LORR}. See also \cite{LP,B2,CPSS}.

The Hardy-Littlewood maximal operator in metric measure spaces has
been mainly studied in the doubling setting (see \cite{H2001}). In
the case of non-doubling spaces, results for a suitable modification
of the maximal operator were provided in \cite{NTV,Sa2005,S2015}.
It is worth mentioning that since Bourgain's seminal work \cite{Bou}
a number of papers{,} such as \cite{SW,MSW,AK2010,CH}{,} have been devoted
to the study of discrete versions of operators in harmonic analysis.

In \cite{NT} Naor and Tao {study} the connection between the doubling
condition and the maximal function in metric measure spaces. They
provide a deep localization theorem for the maximal function and introduce
the $n$-microdoubling property and use it to provide some interesting
consequences. Among them they recover the classical result by Strömberg
and Stein \cite{SS} $\|M\|_{L^{1}\rightarrow L^{1,\infty}}\lesssim n\log(n)$
in the general context of metric spaces satisfying the aforementioned
$n$-microdubling property.

Having in mind results such as the Strömberg and Stein bound mentioned
above, one may tend to think that there is always a connection between
the doubling condition of the space and the weak type $(1,1)$ of
the maximal function. However Naor and Tao show, in some sense, that
it is not the case. They provide an example, the infinite {rooted} $k$-ary
tree, for which even in complete absence of the doubling condition,
the weak-type $(1,1)$ for the centered maximal function {holds} (see
Theorem \ref{ThmTaoNaor} a few lines below). 

Given $k\geq2$ we will {denote by} $T_{k}$ the {infinite rooted} $k$-ary tree, namely,
the infinite {rooted} tree such that each vertex has $k$ children. We shall
drop $k$ and write just $T$ in case there {is} no place to confusion. {Abusing of notation, we will also use $T$ to denote its vertex set.}
It is possible to define a metric measure space $(T,d,\mu)$ where
$d$ is the usual tree metric, namely $d(x,y)$ is the number of edges
of the unique path between $x$ and $y$, and $\mu$ is the counting
measure defined on parts of the set of vertices. Abusing of notation,
given $A\subset T$ we shall denote $|A|=\mu(A)$ and
\[
\int_{A}f(x)dx=\sum_{x\in A}f(x).
\]
We will also denote 
\[
M^{\circ}f(x)=\sup_{r{\geq}0}\frac{1}{|S(x,r)|}\int_{S(x,r)}|f(y)|dy
\]
where $S(x,r)=\{y\in T:d(x,y)=r\}$ denotes the sphere with center
$x$ and radius $r$. Note that, in contrast with the standard {Euclidean}
setting, {here} it makes sense to consider this kind of maximal function
{because} $S(x,r)$ are not sets of measure $0$. For $k\geq2$, we have
that $M^{\circ}\simeq M$ as we will show in Proposition \ref{prop:MMcirc}. 

In the {infinite rooted} $k$-ary tree setting, covering arguments are essentially unavailable
since the doubling condition or even more generally the upper doubling
condition on the measure introduced by Hytönen in \cite{HPub} completely
{fail}. Hence a different approach {is} required. Via a combinatorial
argument, exploiting the ``expander'' {or ``non-amenability''} properties
of the {infinite rooted $k$-ary} tree, Naor and Tao managed to settle the following {theorem}.

\begin{mythm}\label{ThmTaoNaor} {If }$k\geq2$, then
\[
\left|\left\{ x\in T_{k}\,:\,M^{\circ}f(x)>\lambda\right\} \right|\leq\frac{c}{\lambda}\int_{T}|f(x)|dx
\]
with $c$ independent of $k$.

\end{mythm}

It is worth noting that this result can be deduced from a work of
Rochberg and Taibleson~\cite{RT}, and that it was also established
independently by Cowling, Meda, and Setti \cite{CMS}. For $p>1$
the strong type estimate was essentially settled by Nevo and Stein
\cite{NS}. 

At this point we would like to mention works by Soria and Tradacete
\cite{ST,STInf} in which they study the connection between properties
of the maximal function {and} properties of {the} underlying graphs. Furthermore
\cite[Theorem 4.1]{STInf} is an abstract version of Theorem \ref{ThmTaoNaor}.

The main purpose of this work is to get a variant of the Fefferman-Stein
estimate for the Hardy-Littlewood maximal function on the {infinite rooted} $k$-ary
tree, generalizing Theorem \ref{ThmTaoNaor}.  Most of Fefferman-Stein
inequalities in a number of settings rely upon a suitable use of covering
lemmas such as Calderón-Zygmund decomposition. In the {infinite rooted} $k$-ary tree
setting, no regularity nor doubling condition is available, and hence
other techniques are required. We will exploit the flexibility in
the approach provided in \cite{NT} to obtain the following {theorem}.
\begin{thm}
\label{Thm:FS}Let $k\geq2$ and $s>1$. Then, for every weight $w\geq0$
on $T$ we have that 
\[
w\left(\left\{ x\in T\,:\,Mf(x)>\lambda\right\} \right)\leq c_{s}\frac{1}{\lambda}\int_{T}|f(x)|M_{s}w(x)dx
\]
where $M_{s}w=M(w^{s})^{\frac{1}{s}}${,} $c_{s}$ is independent
of $k${,} and $c_{s}\rightarrow+\infty$ when $s\rightarrow1$. 
\end{thm}

At first sight, having in mind the estimate in the classical setting,
one may wonder {whether} this estimate could be improved to match (\ref{eq:FS}).
However, this is not the case. Not only {it} is not possible to choose
$s=1$ but actually {it} is not even possible to choose any number of iterations
of the maximal function for the inequality to hold.
\begin{thm}
\label{thm:NotFS}Let $n\geq1$. There exists a weight $w\sim Mw$
and a sequence $f_{j}\in L^{1}(M^{n}w)$ such that 
\[
w\left(\left\{ Mf_{j}(x)>1\right\} \right)\geq c_{j}\int_{T}|f_{j}{(x)}|M^{n}w{(x)dx}
\]
where $M^{n}=M\circ\stackrel{n\text{ times}}{{\cdots}}\circ M$ and $c_{j}\rightarrow\infty$
when $j\rightarrow\infty$. 
\end{thm}

A direct consequence of the preceding {theorem} is that the fact that a weight $w$
\[
Mw(x)\leq c_{w}w(x)\qquad{\text{for all }}x\in T
\]
is not sufficient for $\|M\|_{L^{1}(w)\rightarrow L^{1,\infty}(w)}<\infty$.

On the other hand we have the following {corollary} of Theorem \ref{Thm:FS}.
\begin{cor}
\label{cor:weak11}Let $w$ {be} a weight such that there exists $s>1$
{for which}
\[
M_{s}w(x)\leq c_{w,s}w(x)\qquad{\text{for all }} x\in T.
\]
Then 
\[
w\left(\left\{ x\in T\,:\,Mf(x)>\lambda\right\} \right)\lesssim\frac{1}{\lambda}\int_{T}|f(x)|w(x)dx.
\]
\end{cor}

We would like to observe that {throughout} the paper we deal with {infinite rooted} $k$-ary
trees with $k\geq2$. It is easy to check that in the case $k=1$,
Theorem \ref{Thm:FS} holds even for $s=1$, since the the measure
on the {infinite rooted} $1$-ary tree is a doubling measure, and hence the classical
theory works. Besides that, {Proposition}~\ref{prop:MMcirc} does not
hold for $k=1${;} furthermore, it is not hard to check that $M^{\circ}$
is not of weak type $(1,1)$ in this case.

The remainder of the paper is organized as follows. Section \ref{sec:FSest}
is devoted to {settling} Theorem~\ref{Thm:FS}. In Section \ref{sec:RW}
we provide examples of non trivial weights that fulfill the {assumptions of}
Corollary \ref{cor:weak11} and we settle Theorem \ref{thm:NotFS}.
Section \ref{sec:VectorValuedExtensions} is devoted to {giving} some
vector valued extensions. We end up the paper with an Appendix devoted
to {providing} a weighted counterpart of \cite[Theorem 4.1]{STInf}.

\section{\label{sec:FSest}Proof of Theorem \ref{Thm:FS}}

Before going into the proof of the theorem we present the following
proposition{.}
\begin{prop}
\label{prop:MMcirc}Let $k\geq2$ and $f\in L^{1}(T_{k})$. Then
\[
Mf(x)\leq M^{\circ}f(x)\leq2Mf(x).
\]
\end{prop}

\begin{proof}
Since every ball can be written as the disjoint union of spheres,
we have the pointwise estimate 
\[
Mf(x)\leq M^{\circ}f(x).
\]
For the other inequality, let $r\in\mathbb{N}${. We} begin observing
that  
\[
\begin{split}\frac{|B(x,r)|}{|S(x,r)|} & =\frac{\sum_{j=0}^{r}|S(x,j)|}{|S(x,r)|}\\
 & \simeq\frac{k^{r}+k^{r-1}+\dots+1}{k^{r}}=\sum_{j=0}^{r}\frac{1}{k^{j}}\leq2.
\end{split}
\]
Hence
\[
\begin{split} & \frac{1}{|S(x,r)|}\int_{S(x,r)}|f(y)|dy\\
 & \leq\frac{|B(x,r)|}{|S(x,r)|}\frac{1}{|B(x,r)|}\int_{B(x,r)}|f(y)|dy\\
 & \leq2\frac{1}{|B(x,r)|}\int_{B(x,r)}|f(y)|dy
\end{split}
\]
and this yields 
\[
M^{\circ}f(x)\leq2Mf(x).\qedhere
\]
\end{proof}
From the previous proposition, if we denote $M_{s}^{\circ}(w)=M^{\circ}(w^{s})^{\frac{1}{s}}$,
it {readily} follows that to settle Theorem \ref{Thm:FS} it suffices
to show that 
\begin{equation}
w\left(\left\{ x\in T\,:\,M^{\circ}f(x)>\lambda\right\} \right)\leq c_{s}\frac{1}{\lambda}\int_{T}|f(x)|M_{s}^{\circ}w(x)dx,\label{treeweak}
\end{equation}
for all $f\in L_{1}(T)$ and $\lambda>0$.

We will denote by $\1\otimes w$ the product measure 
\[
\1\otimes w(A\times B)=|A|w(B)=\sum_{(x,y)\in A\times B}w(y),\qquad A,B\subset T.
\]

The proof of Theorem \ref{Thm:FS} follows the scheme provided
by Naor and Tao \cite{NT}. In particular \cite[Lemma 5.1]{NT} is
a key part of their proof. That lemma is obtained {exploiting}
an expander and combinatorial argument that relies upon the symmetry
of the {infinite rooted} $k$-ary tree. The role played by the fact that the measure
on the space is the counting measure may seem relevant in the proof
to provide a suitable ``sharp'' estimate. However, in the following
lemma we overcome that difficulty providing a weighted version that
contains a precise enough bound that allows us to push the scheme in \cite{NT}.
\begin{lem}
\label{lem:Borders}Let $E,F$ be finite subsets of $T$, $s>1$ and
let $r\geq0$ be an integer. Then
\[
\1\otimes w\left(\{(x,y)\in E\times F:\ d(x,y)=r\}\right)\le c_{s}k^{r\frac{s'}{s'+1}}w(F)^{\frac{1}{s'+1}}M_{s}^{\circ}w(E)^{\frac{s'}{s'+1}}
\]
where $s'=\frac{s}{s-1}$ and $c_{s}$ is a constant depending {only} on
$s$.
\end{lem}

\begin{proof}
We split the {vertex set of the} tree as $T=\bigcup_{j=0}^{\infty}T^{j}$, where $T^{j}$
is the generation of the tree at depth $j$. We split as well accordingly,
$E$ and $F$. We define $E_{j}=E\cap T^{j}$ and $F_{j}=F\cap T^{j}$.
An element in $E_{j}$ and an element in $F_{i}$ can be at distance
exactly $r$, if and only if $i=j+r-2m$ for some $m\in\{0,\ldots,r\}$.
Hence we can write 
\begin{equation}
\begin{split} & \1\otimes w\left(\{(x,y)\in E\times F:d(x,y)=r\}\right)\\
 & =\sum_{m=0}^{r}\sum_{\substack{i,j\in\N\cup\{0\}\\
i=j+r-2m
}
}\1\otimes w\left(\{(x,y)\in E_{j}\times F_{i}:d(x,y)=r\}\right).
\end{split}
\label{eq:levels}
\end{equation}
Now we fix $m\in\{0,\ldots,r\}$ and $i,j\in\N\cup\{0\}$ such that
$i=j+r-2m$. Note that if $x\in T^{j}$ and $y\in T_{i}$ are at distance
$r$ in $T$, then the $m^{th}$ parent of $x$ coincides with the
$(r-m)^{th}$ parent of $y$. This leads to the fact that for each
$y\in T^{i}$ there exist at most $k^{m}$ elements of $x\in T^{j}$
with $d(x,y)=r$. From this it readily follows that
\[
\1\otimes w\left(\{(x,y)\in E_{j}\times F_{i}:d(x,y)=r\}\right)\leq k^{m}w(F_{i}).
\]
On the other hand note that {for} each $x\in T^{j}$ {there} are
at most $k^{r-m}$ elements of $y\in T^{i}$ with $d(x,y)=r${. Hence} we
have that for each $s>1$, 
\[
\begin{split} & \1\otimes w\left(\{(x,y)\in E_{j}\times F_{i}:d(x,y)=r\}\right)\\
 & =\sum_{x\in E_{j}}\sum_{\stackrel{y\in F_{i}}{d(x,y)=r}}w(y)\\
 & =\sum_{x\in E_{j}}w(F_{i}\cap S(x,r))\\
 & \leq\sum_{x\in E_{j}}|F_{i}\cap S(x,r)|^{\frac{1}{s'}}w^{s}(F_{i}\cap S(x,r))^{\frac{1}{s}}\\
 & =\sum_{x\in E_{j}}|F_{i}\cap S(x,r)|^{\frac{1}{s'}}k^{\frac{r}{s}}\frac{1}{k^{\frac{r}{s}}}w^{s}(F_{i}\cap S(x,r))^{\frac{1}{s}}\\
 & \leq\sum_{x\in E_{j}}|F_{i}\cap S(x,r)|^{\frac{1}{s'}}k^{\frac{r}{s}}{M_{s} ^\circ}(w)(x)\\
 & \leq k^{\frac{m-r}{s'}}k^{\frac{r}{s}}{M_{s}^\circ}(w)(E_{j}).
\end{split}
\]
Thus combining the ideas above
\begin{equation}
\1\otimes w\left(\{(x,y)\in E_{j}\times F_{i}:d(x,y)=r\}\right)\leq\min\left\{ k^{r-\frac{m}{s'}}{M_{s}^\circ}w(E_{j}),k^{m}w(F_{i})\right\} .\label{eq:Levelm}
\end{equation}
Taking into account (\ref{eq:levels}) and (\ref{eq:Levelm}), to
end the proof it suffices to show that
\begin{equation}
\sum_{m=0}^{r}\sum_{\substack{i,j\in\N\cup\{0\}\\
i=j+r-2m
}
}\min\left\{ k^{\frac{r}{s}+\frac{r-m}{s'}}{M_{s}^\circ}w(E_{j}),k^{m}w(F_{i})\right\} \le c_{s}k^{r\frac{s'}{s'+1}}w(F_{i})^{\frac{1}{s'+1}}{M_{s}^\circ}w(E_{j})^{\frac{s'}{s'+1}}\label{eq:target}
\end{equation}
Let us define $c_{j}=\frac{M_{s}^\circ w(E_{j})}{k^{\frac{j}{s'}}}$
and $d_{j}=\frac{w(F_{j})}{k^{j}}$ for $j\geq0$ and $c_{j}=d_{j}=0$
for $j<0${. Then}, 
\begin{equation}
\sum_{j=0}^{\infty}k^{\frac{j}{s'}}c_{j}={M_{s}^\circ}w(E)\quad\mathrm{and}\quad\sum_{j=0}^{\infty}k^{j}d_{j}=w(F),\label{eq:redfcjdj}
\end{equation}
and we have that {whenever $i=j+r-2m$,}
\[
\begin{split}\min\left\{ k^{\frac{r}{s}+\frac{r-m}{s'}}{M_{s}^\circ}w(E_{j}),k^{m}w(F_{i})\right\}  & =\min\left\{ k^{\frac{r}{s}+\frac{r-m}{s'}}k^{\frac{j}{s'}}c_{j},k^{m}k^{i}d_{i}\right\} \\
 & =\min\left\{ k^{r-\frac{r}{2s'}}k^{\frac{i+j}{2s'}}c_{j},k^{\frac{r}{2}}k^{\frac{i+j}{2}}d_{i}\right\}{.} 
\end{split}
\]
Taking the identity above into account, settling (\ref{eq:target})
reduces to show that 
\[
\sum_{m=0}^{r}\sum_{\substack{i,j\in\N\cup\{0\}\\
i=j+r-2m
}
}\min\left\{ k^{r-\frac{r}{2s'}}k^{\frac{i+j}{2s'}}c_{j},k^{\frac{r}{2}}k^{\frac{i+j}{2}}d_{i}\right\} \leq c_{s}k^{r\frac{s'}{s'+1}}w(F_{i})^{\frac{1}{s'+1}}{M_{s}^\circ}w(E_{j})^{\frac{s'}{s'+1}}.
\]
To prove this inequality, we fix {a real parameter $\alpha$} to
be chosen later, and {argue} as follows{:} 
\begin{eqnarray*}
 &  & \sum_{m=0}^{r}\sum_{\substack{i,j\in\N\cup\{0\}\\
i=j+r-2m
}
}\min\left\{ k^{r-\frac{r}{2s'}}k^{\frac{i+j}{2s'}}c_{j},k^{\frac{r}{2}}k^{\frac{i+j}{2}}d_{i}\right\} \\
 &  & \leq\sum_{\substack{i,j\in\N\cup\{0\}\\
i<j+\alpha
}
}k^{r-\frac{r}{2s'}}k^{\frac{i+j}{2s'}}c_{j}+\sum_{\substack{i,j\in\N\cup\{0\}\\
i\geq j+\alpha
}
}k^{\frac{r}{2}}k^{\frac{i+j}{2}}d_{i}\\
&  & {= k^{r-\frac{r}{2s'}}\sum_{j=0}^\infty\sum_{i\in\N\cup\{0\}:i<j+\alpha}k^{\frac{i+j}{2s'}}c_{j}+k^{\frac{r}{2}}\sum_{i=0}^\infty\sum_{j\in\N\cup\{0\}:j\leq i-\alpha}k^{\frac{i+j}{2}}d_{i}}\\
 &  & {\lesssim{}} c_{s}k^{r-\frac{r}{2s'}}\sum_{j=0}^{\infty}k^{\frac{j}{s'}+\frac{\alpha}{2s'}}c_{j}+k^{\frac{r}{2}}\sum_{i=0}^{\infty}k^{i-\frac{\alpha}{2}}d_{i}\\
 &  & \le c_{s}\left(k^{r-\frac{r}{2s'}}k^{\frac{\alpha}{2s'}}{M_{s}^\circ}w(E_{j})+k^{\frac{r}{2}}k^{-\frac{\alpha}{2}}w(F_{i})\right){.}
\end{eqnarray*}
Now {we} provide some hints about how to optimize on $\alpha$. Let $f_{a,b}(\alpha)=k^{\frac{\alpha}{2s'}}a+k^{-\frac{\alpha}{2}}b$
for $a,b>0$. Note that $f_{a,b}$ reaches its absolute minimum at
$\frac{2\log_{k}(\frac{b}{a})}{1+\frac{1}{s'}}.$ Hence choosing $a_{0}=k^{r-\frac{r}{2s'}}{M_{s}^\circ}w(E_{j})$
and $b_{0}=k^{\frac{r}{2}}w(F_{i})$ and $\alpha_{0}=\frac{2\log_{k}(\frac{b_{0}}{a_{0}})}{1+\frac{1}{s'}}$
we have that 
\[
\begin{split} \sum_{m=0}^{r}\sum_{\substack{i,j\in\N\cup\{0\}\\
i=j+r-2m
}
}&\min\left\{ k^{r-\frac{r}{2s'}}k^{\frac{i+j}{2s'}}c_{j},k^{\frac{r}{2}}k^{\frac{i+j}{2}}d_{i}\right\}\\
 & \leq c_{s}f_{a_{0},b_{0}}(\alpha_{0})\\
 & \leq c_{s}\left(k^{r\frac{s'}{s'+1}}w(F_{i})^{\frac{1}{s'+1}}{M_{s}^\circ}w(E_{j})^{\frac{s'}{s'+1}}+k^{r\frac{{s'}}{s'+1}}w(F_{i})^{\frac{1}{s'+1}}{M_{s}^\circ}w(E_{j})^{\frac{s'}{s'+1}}\right)\\
 & \leq c_{s}k^{r\frac{s'}{s'+1}}w(F_{i})^{\frac{1}{s'+1}}{M_{s}^\circ}w(E_{j})^{\frac{s'}{s'+1}}
\end{split}
\]
and hence we are done.
\end{proof}
For each $r\geq0$, we denote by $A_{r}^{\circ}$ the spherical averaging
operator 
\[
A_{r}^{\circ}f(x)=\frac{1}{|S(x,r)|}\sum_{y\in S(x,r)}|f(y)|.
\]
Hence $M^{\circ}f(x)=\sup_{r\ge0}A_{r}^{\circ}f(x)$. We can use Lemma
\ref{lem:Borders} to obtain a distributional estimate on $A_{r}^{\circ}$. 
\begin{lem}
\label{lem:sumLevels} Let {$r\geq 1$} and $\lambda>0$. Then
\[
w\left(\left\{ A_{r}^{\circ}f\geq\lambda\right\} \right)\lesssim c_{s}\sum_{\substack{n\in\N\cup\{0\}\\
1\leq2^{n}\leq2k^{r}
}
}\left(\frac{2^{n}}{k^{r}}\right)^{\frac{1}{2s'}}2^{n}{M_{s}^\circ}w\left(\left\{ |f|\geq2^{n-1}\lambda\right\} \right)
\]
where {$c_s$ depends only on $s$ and} $c_{s}\rightarrow\infty$ when $s\rightarrow1$.
\end{lem}

\begin{rem}
We would like to note that the decay $\left(\frac{2^{n}}{k^{r}}\right)^{\frac{1}{2s'}}$
will be fundamental for our purposes. Note that in the case $s=1$
then we would not have this decay and, as we will see later, in the
absence of that decay we would not be able to settle Theorem \ref{Thm:FS}.
At this point we would like to note as well that this inequality with
a good enough decay in $\frac{2^{n}}{k^{r}}$ and $s=1$ cannot hold
since that would contradict Theorem \ref{thm:NotFS}.
\end{rem}

\begin{proof}[Proof of Lemma \ref{lem:sumLevels}]
We can assume without loss of generality $f$ to be non-negative {and $\lambda=1$}.
We bound 
\begin{equation}
f\leq\frac{1}{2}+\sum_{\substack{n\in\N\cup\{0\}\\
1\leq2^{n}\leq k^{r}
}
}2^{n}\chi_{E_{n}}+f\chi_{\{f\geq\frac{1}{2}k^{r}\}},\label{eq:troceado}
\end{equation}
where $E_{n}$ is the sublevel set 
\begin{equation}
E_{n}=\left\{ 2^{n-1}\leq f<2^{n}\right\} .\label{eq:niveles}
\end{equation}
Hence 
\begin{equation}
A_{r}^{\circ}f\leq\frac{1}{2}+\sum_{\substack{n\in\N\cup\{0\}\\
1\leq2^{n}\leq k^{r}
}
}2^{n}A_{r}^{\circ}\left(\chi_{E_{n}}\right)+A_{r}^{\circ}\left(f\chi_{\{f\geq\frac{1}{2}k^{r}\}}\right).\label{eq:troceadopromediado}
\end{equation}
Since $|S(x,r)|\leq k^{r}$ we see that 
\begin{equation}
\begin{split}w\left(A_{r}^{\circ}\left(f\chi_{\{f\geq\frac{1}{2}k^{r}\}}\right)\neq0\right) & \le w\left(\bigcup_{y\in\{f\geq\frac{1}{2}k^{r}\}}{S}(y,r)\right)\\
 & \leq\sum_{y\in\{f\geq\frac{1}{2}k^{r}\}}w({S}(y,r))=|S(x,r)|\sum_{y\in\{f\geq\frac{1}{2}k^{r}\}}\frac{w({S}(y,r))}{|S(x,r)|}\\
 & \leq k^{r}{M^\circ}w\left(\left\{f\geq\frac{1}{2}k^{r}\right\}\right){.}
\end{split}
\label{eq:TerminoFacil}
\end{equation}
Thus we have that combining the estimates above
\[
\begin{split} w\left(A_{r}^{\circ}f\geq1\right)
   & {{}\leq{}} w\left(\sum_{\substack{n\in\N\cup\{0\}\\
1\leq2^{n}\leq k^{r}
}
}2^{n}A_{r}^{\circ}\left(\chi_{E_{n}}\right)\geq\frac{1}{2}\right)+w\left(A_{r}^{\circ}\left(f\chi_{\{f\geq\frac{1}{2}k^{r}\}}\right)\neq0\right)\\
 & \leq w\left(\sum_{\substack{n\in\N\cup\{0\}\\
1\leq2^{n}\leq k^{r}
}
}2^{n}A_{r}^{\circ}\left(\chi_{E_{n}}\right)\geq\frac{1}{2}\right)+k^{r}{M^\circ}w\left(\left\{f\geq\frac{1}{2}k^{r}\right\}\right){.}
\end{split}
\]
Let {$\beta$ be a real parameter such that $0<\beta<1$} to be chosen {later}. Note that if 
\[
\sum_{\substack{n\in\N\cup\{0\}\\
1\leq2^{n}\leq k^{r}
}
}2^{n}A_{r}^{\circ}\left(\chi_{E_{n}}\right)\geq\frac{1}{2}
\]
then we necessarily have {some} $n\in\N$ {such} that $1\leq2^{n}\leq k^{r}$ {for which}
\[
A_{r}^{\circ}\left(\chi_{E_{n}}\right)\geq\frac{2 ^{\beta}-1}{2^{{n+2}}}\left(\frac{2^{n}}{k^{r}}\right)^{\beta}.
\]
Indeed, otherwise we have that
\[
\begin{split}\frac{1}{2} & \leq\sum_{\substack{n\in\N\cup\{0\}\\
1\leq2^{n}\leq k^{r}
}
}2^{n}A_{r}^{\circ}\left(\chi_{E_{n}}\right)\le\frac{2^{\beta}-1}{{4}k^{r\beta}}\sum_{\substack{n\in\N\cup\{0\}\\
1\leq2^{n}\leq k^{r}
}
}2^{\beta n}\\
 & \leq\frac{(2^{\beta}-1)}{{4}k^{r\beta}}\frac{(2^{\beta{(}\log_{2}k^{r}{{}+1)}}-1)}{\left(2^{\beta}-1\right)}=\frac{{2^\beta}k^{r\beta}-1}{{4}k^{r\beta}}<{\frac{2^\beta}{4}}<\frac{1}{2}
\end{split}
\]
which is a contraction. Thus 
\[
w\left(A_{r}^{\circ}f\geq{1}\right)\leq\sum_{\substack{n\in\N\cup\{0\}\\
1\leq2^{n}\leq k^{r}
}
}w(F_{n})+k^{r}{M^\circ}w\left(\left\{{f\geq\frac{1}{2}k^{r}}\right\}\right)
\]
where 
\[
F_{n}=\left\{ A_{r}^{\circ}\left(\chi_{E_{n}}\right)\geq\frac{2^{\beta}-1}{2^{{n+2}}}\left(\frac{2^{n}}{k^{r}}\right)^{\beta}\right\} .
\]
Note that $F_{n}$ is finite and observe that{,} since $A_{r}^{\circ}$
is a selfadjoint operator,
\[
\begin{split} & \frac{1}{k^{r}}\1\otimes w\left(\{(x,y)\in E_{n}\times F_{n}:d(x,y)=r\}\right)\\
 & =\frac{1}{k^{r}}\sum_{x\in E_{n}}\sum_{\stackrel{y\in F_{n}}{d(x,y)=r}}w(y)\simeq\int_{T}\chi_{E_{n}}A_{r}^{\circ}(w\chi_{F_{n}})({x}){dx}=\int_{F_{n}}wA_{r}^{\circ}(\chi_{E_{n}})(y){dy}\\
 & \geq w(F_{n})\frac{2^{\beta}-1}{2^{{n+2}}}\left(\frac{2^{n}}{k^{r}}\right)^{\beta}{.}
\end{split}
\]
Now, using Lemma \ref{lem:Borders}{,}
\[
\frac{1}{k^{r}}\1\otimes w\left(\{(x,y)\in E_{n}\times F_{n}:d(x,y)=r\}\right)\leq c_{s}k^{-\frac{r}{s'+1}}w(F_{n})^{\frac{1}{s'+1}}M_{s}^{\circ}w(E_{n})^{\frac{s'}{s'+1}}.
\]
Hence
\[
\begin{split} w(F_{n})\frac{2^\beta-1}{2^{{n+2}}}&\left(\frac{2^{n}}{k^{r}}\right)^{\beta}\leq c_{s}k^{-\frac{r}{s'+1}}w(F_{n})^{\frac{1}{s'+1}}{M_{s}^\circ}w(E_{n})^{\frac{s'}{s'+1}}\\
\iff & w(F_{n})^{1-\frac{1}{s'+1}}\leq c_{s}\frac{1}{2^{\beta}-1}k^{\beta r-\frac{r}{s'+1}}2^{(1-\beta)n}{M_{s}^\circ}w(E_{n})^{\frac{s'}{s'+1}}\\
\iff & w(F_{n})^{\frac{s'}{s'+1}}\leq c_{s}\frac{1}{2^{\beta}-1}k^{\beta r-\frac{r}{s'+1}}2^{(1-\beta)n}{M_{s}^\circ}w(E_{n})^{\frac{s'}{s'+1}}\\
\iff & w(F_{n})\leq c_{s}\frac{1}{{\left(2^{\beta}-1\right)^{\frac{s'+1}{s'}}}}k^{r\left((s'+1)\beta-1\right)\frac{1}{s'}}2^{\frac{s'+1}{s'}(1-\beta)n}{M_{s}^\circ}w(E_{n})
\end{split}
\]
Choosing $\beta=\frac{1}{2(s'+1)}$ we have that
\[
w(F_{n})\leq c_{s}k^{-\frac{r}{2s'}}2^{\frac{n}{2s'}}2^{n}{M_{s}^\circ}w(E_{n})
        {{}\leq c_s\left(\frac{2^{n}}{k^{r}}\right)^{\frac{1}{2s'}}2^{n}M_s^\circ w\left(\left\{f\geq 2^{n-1}\right\}\right)}.
\]
{Therefore
\begin{align*}
  w(\{A_r^\circ f\geq 1\})
     &\leq c_s\sum_{\substack{n\in\N\cup\{0\}\\1\leq 2^n\leq k^r}}\left(\frac{2^{n}}{k^{r}}
     \right)^{\frac{1}{2s'}}2^{n}M_s^\circ w\left(\left\{f\geq 2^{n-1}\right\}\right)+k^rM^\circ w\left(\left\{f\geq\frac{1}{2}k^{r}\right\}\right).
\end{align*}
Since, in the right-hand side, the second term is dominated by the last term of the summation in the fist term,} this yields the desired conclusion.\end{proof}

Combining the ingredients above we are in the position to settle Theorem
\ref{Thm:FS}{.}
\begin{proof}
As we argued above, it suffices to settle (\ref{treeweak}). Since
$M^{\circ}f=\sup_{r\geq0}A_{r}^{\circ}f$, Lemma~\ref{lem:sumLevels}
implies that 
\[
\begin{split}w\left(M^{\circ}f\geq\lambda\right) & \leq\sum_{r=0}^{\infty}w\left(A_{r}^{\circ}f\geq\lambda\right)\\
 & \leq c_{s}\sum_{r=0}^{\infty}\sum_{\substack{n\in\N\cup\{0\}\\
1\leq2^{n}\leq2k^{r}
}
}\left(\frac{2^{n}}{k^{r}}\right)^{\frac{1}{2s'}}2^{n}{M_{s}^\circ}w\left(|f|\geq2^{n-1}\lambda\right)\\
 & =c_{s}\sum_{x\in T}\sum_{n=0}^{\infty}\left(\sum_{\substack{r\in\N\cup\{0\}\\
k^{r}\geq2^{n-1}
}
}\frac{1}{k^{\frac{r}{2s'}}}\right)2^{n+\frac{n}{2s'}}\chi_{\{|f(x)|\geq2^{n-1}\lambda\}}(x){M_{s}^\circ}w(x)\\
 & \lesssim c_{s}\sum_{x\in T}\sum_{n=0}^{\infty}2^{n}\chi_{\{|f(x)|\geq2^{n-1}\lambda\}}{M_{s}^\circ}w(x)\lesssim c_{s}\sum_{x\in T}\frac{1}{\lambda}|f(x)|{M_{s}^\circ}w(x).
\end{split}
\]
Hence (\ref{treeweak}) holds and the proof of Theorem \ref{Thm:FS}
is complete. 
\end{proof}

\section{Examples of non trivial weights and the failure of the classical
Fefferman-Stein estimate\label{sec:RW}}

\subsection{Radial weights}

A natural way to define a radial weight on {the infinite rooted} $k$-ary tree is the
following. Let us consider 
\[
T=\bigcup_{j=0}^{\infty}T^{j}
\]
where $T{^0}$ {is the set whose only element is} the root of the tree, $T{^1}$ {is the set of vertices that are} children
of the root, and analogously $T^{j}$ is the set of vertices that
are children of vertices in $T{^{j-1}}$. Given this splitting, a radial
weight can be defined as follows
\[
w(x)=\sum_{j}c_{j}\chi_{T^{j}}(x)\qquad c_{j}\geq0.
\]
A natural question is trying to find choices of $c_{j}$ such that
\[
Mw(x)\lesssim w(x).
\]
First of all, note that since $Mw(x)\simeq M^{\circ}w(x)$ it suffices
to study the estimate for the latter. The problem would be to prove
\[
\frac{1}{|S(x,r)|}\sum_{y\in S(x,r)}w(y)\leq\kappa w(x)
\]
for some $\kappa>0$ uniformly on $r$ and on $x$. First note that
$|S(x,r)|\simeq k^{r}$. Now arguing as Naor and Tao{~\cite{NT}} we note that
given $x\in T^{j}$ a vertex at distance exactly $r$ belongs to $T^{i}$
if and only if $i=j+r-2m$ where $m\in\{0,\dots r\}$ and there are
exactly $k^{r-m}$ of such vertices. Hence if $x\in T^{j}$
\[
{\frac{1}{|S(x,r)|}\sum_{y\in S(x,r)}w(y)\simeq{}}\frac{1}{k^{r}}\sum_{y\in S(x,r)}w(y)=\frac{1}{k^{r}}\sum_{m=0}^{r}c_{i}k^{r-m}=\frac{1}{k^{r}}\sum_{m=0}^{r}c_{j+r-2m}k^{r-m}{.}
\]
Given the fact that spheres in this tree grow exponentially, a first
natural choice could be studying the behaviour in the case 
\[
c_{j}=k^{j\beta}.
\]
We shall call 
\[
w_{\beta}(x)=\sum_{j}k^{j\beta}\chi_{T^{j}}(x){.}
\]
Note that 
\[
\begin{split}\frac{1}{k^{r}}\sum_{m=0}^{r}c_{j+r-2m}k^{r-m} & =\frac{1}{k^{r}}\sum_{m=0}^{r}k^{(j+r-2m)\beta}k^{r-m}=k^{j\beta}k^{r\beta}\sum_{m=0}^{r}k^{m(-2\beta-1)}\end{split}{.}
\]
Hence we would need to show that
\[
k^{r\beta}\sum_{m=0}^{r}k^{m(-2\beta-1)}\leq c_{\beta}\]
uniformly on $r$. Note that if 
\[
-2\beta-1>0\iff\beta<-\frac{1}{2}
\]
we have that 
\[
k^{r\beta}\sum_{m=0}^{r}k^{m(-2\beta-1)}\simeq c_{\beta}k^{r\beta+r(-2\beta-1)}=c_{\beta}k^{(-\beta-1)r}{.}
\]
Hence if $\beta\in\left[-1,-\frac{1}{2}\right)${,} $k^{r\beta}\sum_{m=0}^{r}k^{m(-2\beta-1)}\leq c_{\beta}$.
If $-\frac{1}{2}\leq\beta<0$ we have that 
\[
k^{r\beta}\sum_{m=0}^{r}k^{m(-2\beta-1)}\leq k^{r\beta}\sum_{m=0}^{r}k^{m(-2\beta-1)}\leq k^{r\beta}r\leq2^{r\beta}rc_{\beta}.
\]
The case $\beta=0$ is trivial, since it corresponds with having no
weight. 

In the remainder of the cases, namely if $\beta\in\mathbb{R\setminus}[-1,0]$,
the claimed uniform estimate is not available. 

The discusion above can be summarized in the following theorem
\begin{thm}
The radial weight 
\[
w_{\beta}(x)=\sum_{j}k^{j\beta}\chi_{T^{j}}(x)
\]
satisfies $Mw\simeq w$ iff $\beta\in\left[-1,0\right]$. 
\end{thm}

\begin{rem}
Note that if $\beta\in(-1,0]$ by the argument above there exists
$s_{\beta}>1$ such that 
\[
M_{s_{\beta}}w_{\beta}(x)\lesssim w_{\beta}(x)
\]
and $w_{\beta}$ satisfies the {assumptions of} Corollary \ref{cor:weak11}.
\end{rem}

\begin{rem}
The argument used above proves as well that in the case $\beta=-1$
the inequality 
\[
M_{\gamma}w_{-1}(x)\lesssim w_{-1}(x)
\]
does not hold for any $\gamma>1$. Indeed, if $\gamma>1$ we have
that 
\[
\frac{1}{|S(x,r)|}\int_{S(x,r)}w_{-1}(x)^{\gamma}dx\simeq\frac{1}{k^{r}}\sum_{m=0}^{r}k^{-(j+r-2m)\gamma}{k^{r-m}}\simeq k^{-j\gamma}c_{\gamma}k^{(\gamma-1)r}
\]
and on the other hand 
\[
\left(\frac{1}{|S(x,r)|}\int_{S(x,r)}w_{-1}(x)dx\right)^{\gamma}\simeq\left(\frac{1}{k^{r}}\sum_{m=0}^{r}k^{-(j+r-2m)}{k^{r-m}}\right)^{\gamma}\simeq k^{-j\gamma}.
\]
\end{rem}

\subsection{The classical Fefferman-Stein {estimate} does not hold}

In this section we give our proof of Theorem (\ref{thm:NotFS}). Let
$w(x)=\sum_{j=0}^{\infty}\frac{1}{k^{j}}\chi_{T^{j}}(x)$ . As we
showed in the preceding section, for this weight we know that 
\[
M^{n}w\simeq_{n}w.
\]
Let
\[
f_{j}(x)=3\chi_{T^{j}}(x){.}
\]
First we observe that, since $\int_{T}f_{j}(x)w(x){dx}=3${,}
\begin{equation}
\int_{T}f_{j}(x)M^{n}w(x){dx}{{}\simeq{}}c_{n}.\label{eq:fj}
\end{equation}
On the other hand, if $x\in T^{i}$ for $i\leq j$,
\[
\begin{split}M^{\circ}f_{j}(x) & =\sup_{r{\geq0}}\frac{1}{|S(x,r)|}\sum_{y\in B(x,r)}f_{j}(x)\\
 & =3\sup_{r{\geq0}}\frac{|T^{j}\cap B(x,r)|}{|S(x,r)|}\\
 & \geq3\frac{|T^{j}\cap B(x,{j-i})|}{|S(x,{j-i})|}\geq3\frac{k^{{j-i}}}{2k^{{j-i}}}=\frac{3}{2}>1
\end{split}
\]
which in turn implies that
\[
\bigcup_{i=0}^{j}T^{i}\subset\left\{ M^{\circ}f_{j}(x)>1\right\} .
\]
Hence, since $w(T^{i})=1$ for every $0\leq i\leq j$, we have that
\begin{equation}
j\leq\sum_{i=0}^{j}w(T^{i})=w\left(\bigcup_{i=0}^{j}T^{i}\right)\leq w\left(\left\{ M^{\circ}f_{j}(x)>1\right\} \right){.}\label{eq:Levelj}
\end{equation}
The desired conclusion readily follows combining (\ref{eq:fj}) and
(\ref{eq:Levelj}).

\subsection{$Mw\lesssim w$ is necessary}

We end up this section settling the following result.
\begin{thm}
Assume that for a weight $w$ the following estimate holds 
\[
w\left(\left\{ x\in T\,:\,M^{\circ}f(x)>\lambda\right\} \right)\leq c_{w}\frac{1}{\lambda}\sum_{{x}\in T}|f(x)|w(x){.}
\]
{Then} $Mw\lesssim w$.
\end{thm}

\begin{proof}
Let us fix $x_{0}\in T$ and $r>0$. We are going to show that 
\[
\frac{1}{k^{r}}w(S(x_{0},r))\lesssim w(x_{0}).
\]
We begin noting that 
\[
S(x_{0},r)\subset\left\{ x\in T\,:\,M^{\circ}(\delta_{x_{0}})>\frac{1}{2k^{r}}\right\} 
\]
Indeed, note that if $x\in S(x_{0},r)$ {then}
\[
M^{\circ}(\delta_{x_{0}})(x){{}\simeq{}}\sup_{s{\geq}0}\frac{1}{k^{s}}\sum_{z\in S(x,s)}\delta_{x_{0}}(z)=\frac{1}{k^{r}}.
\]
{Hence}
\[
w(S(x_{0},r))\leq w\left(\left\{ x\in T\,:\,M^{\circ}(\delta_{x_{0}})>\frac{1}{2k^{r}}\right\} \right)\leq c_{w}2k^{r}\sum_{x\in T}\delta_{x_{0}}(x)w(x)
\]
and consequently
\[
\frac{w(S(x_{0},r))}{k^{r}}\leq2c_{w}w(x_{0})
\]
which readily implies that 
\[
M^{\circ}w(x)\lesssim w(x)\qquad\text{{for all }}x\in T.\qedhere
\]
\end{proof}

\section{\label{sec:VectorValuedExtensions}Vector valued extensions}

An interesting application that Fefferman and Stein found in \cite{FS1971}
for their two weights estimate were bounds for the following vector
valued extensions
\[
\left(\sum_{j=1}^{\infty}M(f_{j}){}^{q}\right)^{\frac{1}{q}}\qquad1<q<\infty.
\]
Those estimates, besides being an extension of the maximal function,
were a generalization of the Marcinkiewicz operator that consists
in choosing each $f_{j}$ to be a suitable characteristic function. 

In our case we will also be able to provide some vector valued extensions.
First we provide $L^{p}$ versions of our endpoint Fefferman-Stein
estimate that are a direct consequence of the fact that 
\[
\|Mf\|_{L^{\infty}(w)}\lesssim\|f\|_{L^{\infty}(M_{s}w)}
\]
combined with Theorem \ref{Thm:FS} and the Marcinkiewicz interpolation
theorem.
\begin{thm}
Let $1<p,s<\infty$ and $w$ a weight. Then
\begin{equation}
\left(\int_{T}\left(Mf\right)^{p}wdx\right)^{\frac{1}{p}}\leq c_{s}\left(\int_{T}\left|f\right|^{p}M_{s}wdx\right)^{\frac{1}{p}}\label{eq:FSp}
\end{equation}
where $c_{s}\rightarrow\infty$ when $s\rightarrow1$.
\end{thm}

At this point we would like to note that this Theorem can be regarded
as a generalization of Nevo and Stein \cite{NS} where the case $w=1$
was essentially settled.

In our next Theorem we provide some vector valued extensions following
classical ideas in~\cite{FS1971}.
\begin{thm}
Let $1<q\leq p<\infty${. Then} 
\[
\left\Vert \left(\sum_{j=1}^{\infty}M(f_{j}){}^{q}\right)^{\frac{1}{q}}\right\Vert _{L^{p}(T)}\leq c_{p,q}\left\Vert \left(\sum_{j=1}^{\infty}|f_{j}|^{q}\right)^{\frac{1}{q}}\right\Vert _{L^{p}(T)}.
\]
\end{thm}

\begin{proof}
If $p=q$ the proof is straightforward, hence we omit it. For the
case $q<p$ we argue by duality. 
\[
\left\Vert \left(\sum_{j=1}^{\infty}M(f_{j}){}^{q}\right)^{\frac{1}{q}}\right\Vert _{L^{p}(T)}^{q}=\sup_{\|g\|_{L^{\left(\frac{p}{q}\right)'}(T)}=1}\left|\int_{T}\sum_{j=1}^{\infty}M(f_{j}){}^{q}gdx\right|.
\]
Note that using~\eqref{eq:FSp}
\[
\begin{split}\left|\int_{T}\sum_{j=1}^{\infty}M(f_{j}){}^{q}gdx\right| & \leq\sum_{j=1}^{\infty}\int_{T}M(f_{j}){}^{q}\left|g\right|dx\\
 & \leq c_{s}^{q}\int_{T}\sum_{j=1}^{\infty}|f_{j}|^{q}M_{s}(g)dx\\
 & \leq c_{s}^{q}\left\Vert \left(\sum_{j=1}^{\infty}|f_{j}|^{q}\right)^{\frac{1}{q}}\right\Vert _{L^{p}(T)}^{q}\left\Vert M_{s}g\right\Vert _{L^{\left(\frac{p}{q}\right)'}(T)}.
\end{split}
\]
Now choosing $s<\left(\frac{p}{q}\right)'$ we have that 
\[
\begin{split} & \left\Vert \left(\sum_{j=1}^{\infty}|f_{j}|^{q}\right)^{\frac{1}{q}}\right\Vert _{L^{p}(T)}^{q}\left\Vert M_{s}g\right\Vert _{L^{\left(\frac{p}{q}\right)'}(T)}\\
 & \leq c_{\left(\frac{p}{q}\right)'}\left\Vert \left(\sum_{j=1}^{\infty}|f_{j}|^{q}\right)^{\frac{1}{q}}\right\Vert _{L^{p}(T)}^{q}\|g\|_{L^{\left(\frac{p}{q}\right)'}(T)}.
\end{split}
\]
Combining the estimates above we are done.
\end{proof}

\appendix

\section{A weighted version of Soria-Tradacete result for infinite trees}

In this appendix we provide a weighted {version of} \cite[Theorem 4.1]{STInf}.
Analogously {to} the case of the {infinite rooted} $k$-ary tree the spherical maximal
function on {any} tree $T$ can be defined as follows{:}
\[
M^{\circ}f(x)=\sup_{r\geq0}\frac{1}{|S(x,r)|}\sum_{y\in S(x,r)}|f(y)|,
\]
where $S(x,r)$ is the sphere 
\[
S(x,r)=\{y\in T:d(x,y)=r\}.
\]

Let ${\alpha}\in(0,1)$ and $s\in(1,\infty)$. We define 

\[
\mathcal{E}_{T}^{w}(s,r,\alpha)=\sup_{\stackrel{{\scriptstyle E,F\subset G}}{|E|,|F|<\infty}}\frac{1}{w(F)^{\alpha}M_{s}^{\circ}w(E)^{1-\alpha}}\sum_{x\in E}\frac{w\left(F\cap S(x,r)\right)}{|S(x,r)|}
\]
and
\[
S_{T}(r)=\sup_{x\in T}|S(x,r)|.
\]
With {these} quantities at our disposal we are ready to settle our weighted
version of \cite[Theorem 4.1]{STInf}.
\begin{thm}
\label{Thm:weak11abs}For every weight $w$ on {a tree} $T$ we have that 
\[
w\left(\left\{ x\in T\,:\,M^{\circ}f(x)>\lambda\right\} \right)\lesssim\Gamma_{T,w,r,\alpha,s}\frac{1}{\lambda}\int_{T}|f(x)|M_{s}^{\circ}w(x)dx
\]
{where}
\[
\Gamma_{T,w,r,\alpha,s}=c_{\alpha}\sup_{n\in\mathbb{N}}\left\{ \sum_{S_{T}(r)\geq2^{n-1}}^{\infty}\mathcal{E}_{T}^{w}(s,r,\alpha)^{\frac{1}{1-\alpha}}S_{T}(r)^{\frac{1}{2}\frac{\alpha}{1-\alpha}}2^{n\frac{1}{2}\frac{\alpha}{1-\alpha}}\right\} 
\]
{and} $c_{\alpha}\rightarrow+\infty$ when $\alpha\rightarrow0$. 
\end{thm}

 Again, as we did for the {infinite rooted} $k$-ary tree, given any infinite tree
$T$ we can define the average operator over the tree $T$ as
\[
A_{r}^{\circ}f(x)=\frac{1}{|S(x,r)|}\sum_{y\in S(x,r)}|f(y)|{.}
\]
Our next lemma contains the key estimate required to settle Theorem
\ref{Thm:weak11abs}.
\begin{lem}
\label{lem:sumLevelsabs} Let $r,s>0$ and $\lambda>0$. Then 
\[
w\left(\left\{ A_{r}^{\circ}f\geq\lambda\right\} \right)\lesssim c_{\alpha}\sum_{n=0}^{n(r)}2^{\frac{n}{2}\frac{\alpha}{1-\alpha}}\mathcal{E}_{T}^{w}(s,r,\alpha){}^{\frac{1}{1-\alpha}}S_{T}(r)^{\frac{1}{2}\frac{\alpha}{1-\alpha}}2^{n}M_{s}^{\circ}w\left(\left\{ |f|\geq2^{n-1}\lambda\right\} \right)
\]
where {$n(r)$ is an integer such that} $2^{n(r)}\leq S_{T}(r){{}<{}}2^{n(r)+1}$ and $c_{\alpha}\rightarrow+\infty$
when $\alpha\rightarrow0$. 
\end{lem}

\begin{proof}
We can assume without loss of generality $f$ to be non-negative {and $\lambda=1$}.
We bound 
\begin{equation}
f\leq\frac{1}{2}+\sum_{n=0}^{n(r)}2^{n}\chi_{E_{n}}+f\chi_{\{f\geq\frac{1}{2}{S_T(r)}\}},\label{eq:troceado-1}
\end{equation}
where $E_{n}$ is the sublevel set 
\begin{equation}
E_{n}=\left\{ 2^{n-1}\leq f<2^{n}\right\} .\label{eq:niveles-1}
\end{equation}
Hence 
\begin{equation}
A_{r}^{\circ}f\leq\frac{1}{2}+\sum_{n=0}^{n(r)}2^{n}A_{r}^{\circ}\left(\chi_{E_{n}}\right)+A_{r}^{\circ}\left(f\chi_{\{f\geq\frac{1}{2}{S_T(r)}\}}\right).\label{eq:troceadopromediado-1}
\end{equation}
First we note that
\begin{equation}
\begin{split}w\left(A_{r}^{\circ}\left(f\chi_{\{f\geq\frac{1}{2}{S_T(r)}\}}\right)\neq0\right) & \le w\left(\bigcup_{y\in\{f\geq\frac{1}{2}{S_T(r)}\}}{S}(y,r)\right)\\
 & \leq\sum_{y\in\{f\geq\frac{1}{2}{S_T(r)}\}}w({S}(y,r))\leq S_{T}(r)\sum_{y\in\{f\geq\frac{1}{2}{S_T(r)}\}}\frac{w(S(y,r))}{|S(y,r)|}\\
 & \leq S_{T}(r)M^\circ w\left(\left\{f\geq\frac{1}{2}{S_T(r)}\right\}\right)
\end{split}
\label{eq:TerminoFacil-1}
\end{equation}
Thus we have that combining the estimates above
\[
\begin{split} w\left(A_{r}^{\circ}f\geq1\right)
 & {{}\leq{}}w\left(\sum_{n=0}^{n(r)}2^{n}A_{r}^{\circ}\left(\chi_{E_{n}}\right)\geq\frac{1}{2}\right)+w\left(A_{r}^{\circ}\left(f\chi_{\{f\geq\frac{1}{2}{S_T(r)}\}}\right)\neq0\right)\\
 & \leq w\left(\sum_{n=0}^{n(r)}2^{n}A_{r}^{\circ}\left(\chi_{E_{n}}\right)\geq\frac{1}{2}\right)+S_{T}(r)M^\circ w\left(\left\{f\geq\frac{1}{2}{S_T(r)}\right\}\right)
\end{split}
\]
Let {$\gamma$ be a real parameter such that $0<\gamma<1$} to be chosen {later}. Note that if 
\[
\sum_{n=0}^{n(r)}2^{n}A_{r}^{\circ}\left(\chi_{E_{n}}\right)\geq\frac{1}{2}
\]
then we necessarily have {some} $n\in\N$, such that $1\leq2^{n}\leq2^{n(r)}$, {for which}
\[
A_{r}^{\circ}\left(\chi_{E_{n}}\right)\geq\frac{2^{\gamma}-1}{2^{{n+2}}}\left(\frac{2^{n}}{S_{{T}}(r)}\right)^{\gamma}.
\]
Indeed, otherwise we have that
\[
\begin{split}\frac{1}{2} & {{}\leq{}}\sum_{n=0}^{n(r)}2^{n}A_{r}^{\circ}\left(\chi_{E_{n}}\right)\le\frac{\left(2^{\gamma}-1\right)}{{4}S_{T}(r)^{\gamma}}\sum_{n=0}^{n(r)}2^{\gamma n}\\
 & \leq\frac{\left(2^{\gamma}-1\right)}{{4}S_{T}(r)^{\gamma}}\frac{(2^{\gamma{(}n(r){+1)}}-1)}{\left(2^{\gamma}-1\right)}\leq\frac{{2^\gamma}S_{T}(r)^{\gamma}-1}{{4}S_{T}(r)^{\gamma}}{{}<\frac{2^\gamma}{4}<{}}\frac{1}{2}
\end{split}
\]
which is a contraction. Thus 
\[
w\left(A_{r}^{\circ}f\geq{1}\right)\leq\sum_{n=0}^{n(r)}w(F_{n})+{S_T(r)}M^{\circ}w\left({\left\{f\geq\frac{1}{2}S_T(r)\right\}}\right)
\]
where 
\[
F_{n}=\left\{ A_{r}^{\circ}\left(\chi_{E_{n}}\right)\geq\frac{2^{\gamma}-1}{2^{{n+2}}}\left(\frac{2^{n}}{S_{T}(r)}\right)^{\gamma}\right\} .
\]
Note that $F_{n}$ is finite and observe that since $A_{r}^{\circ}$
is a selfadjoint operator,
\[
\begin{split}w(F_{n})\frac{2^{\gamma}-1}{2^{{n+2}}}\left(\frac{2^{n}}{S_{T}(r)}\right)^{\gamma} & \leq\int_{F_{n}}wA_{r}^{\circ}(\chi_{E_{n}})(y){dy}=\int_{E_{n}}A_{r}^{\circ}(w\chi_{F_{n}})(y){dy}\\
 & =\sum_{x\in E_{n}}A_{r}^{\circ}(w\chi_{F_{n}})=\sum_{x\in E_{n}}\frac{1}{|S(x,r)|}\sum_{\stackrel{y\in F_{n}}{d(x,y)=r}}w(y)\\
 & =\sum_{x\in E_{n}}\frac{w(F_{n}\cap S(x,r))}{|S(x,r)|}\\
 & {{}\le{}}\mathcal{E}_{T}^{w}(s,r,\alpha)w(F_{n})^{\alpha}M_{s}^{\circ}w(E_{n})^{1-\alpha}.
\end{split}
\]
Now we observe that 
\[
\begin{split} w(F_{n})\frac{2^{\gamma}-1}{2^{{n+2}}}&\left(\frac{2^{n}}{S_{T}(r)}\right)^{\gamma}\le\mathcal{E}_{T}^{w}(s,r,\alpha)w(F_{n})^{\alpha}M_{s}^{\circ}w(E_{n})^{1-\alpha}\\
\iff & w(F_{n})^{1-\alpha}\lesssim\frac{1}{2^{\gamma}-1}2^{n-\gamma n}\mathcal{E}_{T}^{w}(s,r,\alpha)S_{T}(r)^{\gamma}M_{s}^{\circ}w(E_{n})^{1-\alpha}\\
{\iff} & w(F_{n})\lesssim\frac{1}{{(2^{\gamma}-1)^{\frac 1{1-\alpha}}}}2^{n\frac{1-\gamma}{1-\alpha}}\mathcal{E}_{T}^{w}(s,r,\alpha)^{\frac{1}{1-\alpha}}S_{T}(r)^{\frac{\gamma}{1-\alpha}}M_{s}^{\circ}w(E_{n}).
\end{split}
\]
If we choose $\gamma=\frac{\alpha}{2}$  then 
\[
w(F_{n})\lesssim c_{\alpha}2^{n\frac{1}{2}\frac{\alpha}{1-\alpha}}\mathcal{E}_{T}^{w}(s,r,\alpha)^{\frac{1}{1-\alpha}}S_{T}(r)^{\frac{1}{2}\frac{\alpha}{1-\alpha}}M_{s}^{\circ}w(E_{n})2^{n}
\]
with $c_{\alpha}\rightarrow\infty$ when $\alpha\rightarrow0$. {This leads to the desired conclusion as in the end of the proof of Lemma~\ref{lem:sumLevels}.}
\end{proof}
Having the above {lemma} at our disposal we are in the position to
settle Theorem \ref{Thm:weak11abs}.
\begin{proof}
Since $M^{\circ}f=\sup_{r\geq0}A_{r}^{\circ}f$, Lemma~\ref{lem:sumLevelsabs}
implies that 
\[
\begin{split} & w\left(M^{\circ}f\geq\lambda\right)\\
 & \leq\sum_{r=0}^{\infty}w\left(A_{r}^{\circ}f\geq\lambda\right)\\
 & \lesssim c_{\alpha}\sum_{r=0}^{\infty}\sum_{n=0}^{n(r)}\mathcal{E}_{T}^{w}(s,r,\alpha)^{\frac{1}{1-\alpha}}S_{T}(r)^{\frac{1}{2}\frac{\alpha}{1-\alpha}}2^{n\frac{1}{2}\frac{\alpha}{1-\alpha}}2^{n}M_{s}^{\circ}w\left(\left\{ |f|\geq2^{n-1}\lambda\right\} \right)\\
 & \lesssim c_{\alpha}\sum_{n=0}^{\infty}\sum_{S_{T}(r)\geq2^{n-1}}^{\infty}\mathcal{E}_{T}^{w}(s,r,\alpha)^{\frac{1}{1-\alpha}}S_{T}(r)^{\frac{1}{2}\frac{\alpha}{1-\alpha}}c_{s}2^{n\frac{1}{2}\frac{\alpha}{1-\alpha}}2^{n}M_{s}^{\circ}w\left(\left\{ |f|\geq2^{n-1}\lambda\right\} \right)\\
 & \lesssim c_{\alpha}\sup_{n\in\mathbb{N}}\left\{ \sum_{S_{T}(r)\geq2^{n-1}}^{\infty}\mathcal{E}_{T}^{w}(s,r,\alpha)^{\frac{1}{1-\alpha}}S_{T}(r)^{\frac{1}{2}\frac{\alpha}{1-\alpha}}c_{s}2^{n\frac{1}{2}\frac{\alpha}{1-\alpha}}\right\} \sum_{x\in T}\sum_{n=0}^{\infty}2^{n}\chi_{\{|f(x)|\geq2^{n-1}\lambda\}}M_{s}^{\circ}w(x)\\
 & \lesssim c_{\alpha}\sup_{n\in\mathbb{N}}\left\{ \sum_{S_{T}(r)\geq2^{n-1}}^{\infty}\mathcal{E}_{T}^{w}(s,r,\alpha)^{\frac{1}{1-\alpha}}S_{T}(r)^{\frac{1}{2}\frac{\alpha}{1-\alpha}}c_{s}2^{n\frac{1}{2}\frac{\alpha}{1-\alpha}}\right\} \sum_{x\in T}\frac{1}{\lambda}|f(x)|M_{s}^{\circ}w(x).
\end{split}
\]
This proves Theorem \ref{Thm:weak11abs}.
\end{proof}
\begin{rem*}
We would like to note that computing $\mathcal{E}_{T}^{w}(s,r,\alpha)$
and $S_{T}(r)$ and choosing a suitable $\alpha$ can be very difficult{. However} in certain cases such as {for} the {infinite rooted} $k$-ary tree $T$ it is possible.
First we recall that as was noted in \cite{STInf}, 
\[
S_{T}(r)\simeq k^{r}.
\]
Now, from Lemma \ref{lem:Borders} it follows that
\[
\begin{split}\sum_{x\in E}\frac{w(F\cap S(x,r))}{|S(x,r)|}\lesssim c_{s}S_{T}(r)^{\left(\frac{s'}{s'+1}-1\right)}w(F_{i})^{\frac{1}{s'+1}}M_{s}^{\circ}w(E_{j})^{\frac{s'}{s'+1}}\end{split}
\]
and this yields
\[
\frac{1}{w(F_{i})^{\frac{1}{s'+1}}M_{s}^{\circ}w(E_{j})^{\frac{s'}{s'+1}}}\sum_{x\in E}\frac{w(F\cap S(x,r))}{|S(x,r)|}\lesssim c_{s}S_{T}(r)^{-\frac{1}{s'+1}}.
\]
Hence, choosing $\alpha=\frac{1}{s'+1}$ and consequently $1-\alpha=\frac{s'}{s'+1}$
we have that 
\[
\mathcal{E}_{T}^{w}(s,r,\alpha)\leq c_{s}k^{-\frac{r}{s'+1}}{.}
\]
Then, since $\frac{\alpha}{1-\alpha}=\frac{1}{s'}${,}
\[
\mathcal{E}_{T}^{w}(s,r,\alpha)^{\frac{1}{1-\alpha}}S_{T}(r)^{\frac{1}{2}\frac{\alpha}{1-\alpha}}2^{n\frac{1}{2}\frac{\alpha}{1-\alpha}}\lesssim c_{s}k^{-\frac{r}{s'}}k^{r\frac{1}{2}\frac{1}{s'}}2^{\frac{n}{2}\frac{1}{s'}}\simeq c_{s}\left(\frac{2^{n}}{k^{r}}\right)^{\frac{1}{2s'}}
\]
and consequently, Theorem \ref{Thm:weak11abs} recovers the estimate
in Theorem \ref{Thm:FS}. Note that, by the definition of $\alpha${, in}
this case $c_{\alpha}$ is actually $\tilde{c}_{s}$ and $\tilde{c}_{s}\rightarrow\infty$
when $s\rightarrow1$ since the latter implies that $\alpha\rightarrow0$.
\end{rem*}
\begin{rem*}
We would like to end this Appendix observing that the definition of
$\mathcal{E}_{T}^{w}(s,r,\alpha)$ could be stated as follows
\[
\mathcal{E}_{T}^{w}(\tilde{M},r,\alpha)=\sup_{\stackrel{{\scriptstyle E,F\subset G}}{|E|,|F|<\infty}}\frac{1}{w(F)^{\alpha}\tilde{M}w(E)^{1-\alpha}}\sum_{x\in E}\frac{w\left(F\cap S(x,r)\right)}{|S(x,r)|}.
\]
where $\tilde{M}$ is some maximal operator such that $M^\circ g\lesssim\tilde{M}g$
for any function $g\in L^{1}(T)$. Then, exactly the same argument
given above allows us to prove the following more general estimate
\[
w\left(\left\{ x\in T\,:\,M^{\circ}f(x)>\lambda\right\} \right)\lesssim\Gamma_{T,w,r,\alpha,\tilde{M}}\frac{1}{\lambda}\int_{T}|f(x)|\tilde{M}w(x)dx.
\]
where
\[
\Gamma_{T,w,r,\alpha,\tilde{M}}=c_{\alpha}\sup_{n\in\mathbb{N}}\left\{ \sum_{S_{T}(r)\geq2^{n-1}}^{\infty}\mathcal{E}_{T}^{w}(\tilde{M},r,\alpha)^{\frac{1}{1-\alpha}}S_{T}(r)^{\frac{1}{2}\frac{\alpha}{1-\alpha}}c_{s}2^{n\frac{1}{2}\frac{\alpha}{1-\alpha}}\right\} 
\]
with $c_{\alpha}\rightarrow\infty$ as $\alpha\rightarrow0$.
\end{rem*}
\begin{acknowledgement*}
The first author would like to thank Javier Soria for sharing with
him some insights about endpoint estimates for the maximal function
on graphs that were one of the motivations to carry out this work.
\end{acknowledgement*}
\bibliographystyle{plain}
\bibliography{refs}

\end{document}